\documentclass[12pt, reqno, a4paper]{amsart}

\usepackage{ amssymb, amsmath, enumerate, amsfonts, amsthm, mathrsfs, url, bm, mathtools}

\usepackage{xcolor}
\usepackage{hyperref}
\hypersetup{
colorlinks,
   linkcolor={cyan!80!black},
   citecolor={cyan!80!black},
 urlcolor={cyan!80!black}
}

\usepackage{color}

\usepackage[margin=1in]{geometry}

\RequirePackage{doi}

\usepackage{amscd}
\usepackage{amsfonts}
\usepackage{float}
\usepackage{color}
\usepackage[
backend=biber,
style=alphabetic,
]{biblatex}
\usepackage{bookmark}

\renewbibmacro{in:}{}
\DeclareFieldFormat{title}{#1}

\DeclareFieldFormat[article]{title}{\mkbibemph{#1}}       
\DeclareFieldFormat[incollection]{title}{\mkbibemph{#1}}  
\DeclareFieldFormat[book]{title}{\mkbibemph{#1}}          
\DeclareFieldFormat[incollection]{booktitle}{#1}          
\DeclareFieldFormat[article]{journaltitle}{#1}            

\AtEveryBibitem{%
  \ifentrytype{misc}{\DeclareFieldFormat{title}{\mkbibemph{#1}}}{}}

\DeclareFieldFormat{eprint:eprint}{arXiv:\href{https://arxiv.org/abs/#1}{#1}}

\DeclareFieldFormat[inproceedings]{title}{\mkbibemph{#1}}

\DeclareFieldFormat[inproceedings]{booktitle}{#1}

\usepackage{amssymb}

\newtheorem{theorem}{Theorem}[section]
\newtheorem{lemma}{Lemma}[section]

\newtheorem{proposition}{Proposition}[section]

\theoremstyle{definition}

\theoremstyle{remark}
\newtheorem{remark}{Remark}[section]

\numberwithin{equation}{section}

\newcommand{\Mod}[1]{\ (\mathrm{mod}\ #1)}
\newcommand{\C}{\mathbb C}
\newcommand{\Q}{\mathbb Q}
\newcommand{\Z}{\mathbb Z}
\newcommand{\D}{\mathcal D}
\newcommand{\B}{\mathcal B}
\newcommand{\sym}{\operatorname{sym}}
\newcommand{\Tr}{\operatorname{Tr}}
\renewcommand{\Re}{\operatorname{Re}}
\renewcommand{\Im}{\operatorname{Im}}
\renewcommand{\leq}{\leqslant}
\renewcommand{\geq}{\geqslant}

\begin{document}

\title[Extreme values of central $L$-derivatives and heights of Heegner points]
{Extreme values of central $L$-derivatives and heights of Heegner points}

\author{Mohammad H. Hamdar}
\address{Department of Mathematics \& Statistics, Concordia University,
Montreal, Quebec, Canada\\
And\\
D\'epartement de math\'ematiques et de statistique\\
Universit\'e de Montr\'eal\\Montr\'eal, Quebec\\Canada}
\email{hamdarmohammad98@gmail.com}

\author{Sun-Kai Leung}
\address{Mathematical Institute, University of Oxford, Andrew Wiles Building\\
Radcliffe Observatory Quarter, Woodstock Rd\\
Oxford OX2 6GG\\United Kingdom}
\email{sunkaileung@gmail.com}

\date{}

\begin{abstract}
Using Soundararajan's resonance method, we obtain extreme values of the
central derivatives
\[
L^{(r)}(1/2,f\otimes\chi_d),
\]
where $f$ is a normalized newform and $\chi_d$ ranges over primitive
real characters associated with negative fundamental discriminants. The twisted first moment estimate used in the proof may be of independent interest. 
Applying Gross--Zagier-type formulas, we obtain correspondingly
large N\'eron--Tate heights of Heegner points on elliptic curves and large absolute values of Beilinson--Bloch heights of Heegner cycles attached to modular forms.
\end{abstract}

\maketitle
\enlargethispage{3pt}

\section{Introduction}

For $X\geq1$, denote the set of fundamental discriminants of absolute value
between $X$ and $2X$ by
\begin{align*}
\D(X):=\{d\in\Z:X\leq |d|\leq2X,\ d\text{ is a fundamental discriminant}\}.
\end{align*}
For a fundamental discriminant $d$, write
$\chi_d(\cdot):=(\frac d\cdot)$ for the associated primitive quadratic
character. Soundararajan \cite{MR2425151} introduced the resonance
method and showed that, for every fixed $0<C<1/\sqrt5$ and sufficiently
large $X$, we have
\begin{align*}
\max_{d\in\D(X)}|L(1/2,\chi_d)|
\geq\exp\left(C\sqrt{\frac{\log X}{\log\log X}}\right).
\end{align*}
Assuming the generalized Riemann hypothesis for Dirichlet $L$-functions, Darbar and Maiti
\cite{MR4958486} improved this to
\begin{align*}
\max_{d\in\D(X)}|L(1/2,\chi_d)|
\geq\exp\left(C\sqrt{
\frac{\log X\log\log\log X}{\log\log X}}\right)
\end{align*}
for every fixed $0<C<1/2;$ see \cite{dong2026extremevaluesquadraticdirichlet} for $C<1$. Both of these statements allow fundamental discriminants of
either sign.

The resonance method also applies to quadratic twists of modular
$L$-functions. Gun, Kohnen, and Soundararajan \cite{MR4775033} obtained
large central values in their study of Fourier coefficients of
half-integral-weight modular forms. Hua and Huang \cite{MR4670148}
treated quadratic twists of elliptic curves by almost prime discriminants.
In this paper, we consider central derivatives of arbitrary fixed order,
with the twists restricted to negative fundamental discriminants in a
prescribed congruence class.

Given an even integer $k\geq2,$ let
$f\in S_k^*(\Gamma_0(N))$ be a normalized primitive Hecke eigenform of
weight $k,$ level $N$, and trivial nebentypus, which
admits a Fourier expansion at the cusp $\infty$ given by
\begin{align*}
f(z)=\sum_{n\geq1}\lambda_f(n)n^{\frac{k-1}{2}}e(nz)
\qquad \text{for $z \in \mathbb{H}.$}
\end{align*}
We use the unitary normalization
\begin{align*}
L(s,f):=\sum_{n\geq1}\frac{\lambda_f(n)}{n^s},
\qquad
L(s,f\otimes\chi_d):=
\sum_{n\geq1}\frac{\lambda_f(n)\chi_d(n)}{n^s}
\qquad \text{for $\Re(s)>1$},
\end{align*}
so that the central point is $s=1/2.$

Following the congruence restrictions in Radziwi{\l}{\l}--Soundararajan
\cite[Section~2]{MR3425386}, let $N_0:=\operatorname{lcm}(8,N)$.
We call a residue class $a\Mod{N_0}$ \textit{admissible} if
$(a,N_0)=1$ and $a\equiv1$ or $5\Mod8$. For such a class, define
\begin{gather*}
\D_a^{-}:=\{d<0:d\text{ is a fundamental discriminant},\ d\equiv a\Mod{N_0}\},\\
\D_a^{-}(X):=\D_a^{-}\cap\D(X).
\end{gather*}
In particular, all fundamental discriminants in $\D_a^{-}(X)$ are negative and coprime
to $2N$. The root number is constant on this family and is denoted by
$w_a$. Indeed, for $(d,N)=1$, we have
\begin{equation}\label{eq:root-number}
w(f\otimes\chi_d)=w(f)\chi_d(-N);
\end{equation}
see \cite{AL78} and \cite[Section~1]{MR3425386}.

Using Soundararajan's resonance method, we obtain extreme values of $L^{(r)}(1/2,f\otimes\chi_d).$

\begin{theorem}\label{thm:main}
Let $r\geq1$ be an integer. For all admissible $a\Mod{N_0}$ and sufficiently large $X$,
we have
\begin{align*}
\max_{d\in\D_a^{-}(X)}
|L^{(r)}(1/2,f\otimes\chi_d)|
\gg_{f,r}(\log X)^r
\exp\left((1-o(1))\sqrt{\frac{\log X}{\log\log X}}\right).
\end{align*}
For $r=0$, the same
conclusion holds provided that $w_a=+1$.
\end{theorem}

\begin{remark}
The condition $w_a=+1$ is necessary when $r=0$. For example, the
Ramanujan form $\Delta\in S_{12}^*(\mathrm{SL}_2(\Z))$ has root number
$+1$. Hence $w(\Delta\otimes\chi_d)=\chi_d(-1)=-1$ for every
fundamental discriminant $d<0$, and therefore
$L(1/2,\Delta\otimes\chi_d)=0$.
\end{remark}

Given an elliptic curve $E/\Q$ of conductor $N$ and analytic rank at most
one, write $r_{\rm an}(E):=\operatorname{ord}_{s=1/2}L(s,E)$. Let
$f_E\in S_2^*(\Gamma_0(N))$ be the associated primitive Hecke eigenform.
For $d\in\D_1^{-}$, let
$P_d\in E(\Q(\sqrt d))$ denote the \textit{traced Heegner point} associated with
a fixed modular parametrization, as defined in Section~\ref{sec:GZ},
and let $\widehat h$ denote the N\'eron--Tate height on $E$.

Applying the classical Gross--Zagier formula, we obtain
large N\'eron--Tate heights $\widehat h(P_d).$

\begin{theorem}\label{thm:heegner}
For all sufficiently
large $X$, we have
\begin{align*}
\max_{d\in\D_1^{-}(X)}\widehat h(P_d)
\gg_E\sqrt X(\log X)^{1-r_{\rm an}(E)}\exp\left((1-o(1))\sqrt{\frac{\log X}{\log\log X}}\right).
\end{align*}
\end{theorem}

Given a Hecke eigenform $f\in S_k^*(\Gamma_0(N))$ of analytic rank $r_f:=\operatorname{ord}_{s=1/2} L(s,f)$ at most one and a discriminant $d\in\D_1^{-}$, let $\Delta_{f,d}$ be the
normalized $f$-isotypic traced Heegner cycle defined in
Section~\ref{sec:Zhang}, and let $\langle\ ,\ \rangle_{\rm{BB}}$ denote
the Beilinson--Bloch height pairing with the normalization specified there.
Applying Zhang's Gross--Zagier formula \cite{shouwuzhang}, we similarly obtain
large absolute values of these heights.

\begin{theorem}\label{thm:heegnerCycle}
For all sufficiently
large $X$, we have
\begin{align*}
\max_{d\in\D_1^{-}(X)}
\left|\langle\Delta_{f,d},\Delta_{f,d}\rangle_{\rm{BB}}\right|
\gg_f\sqrt X (\log X)^{1-r_f}\exp\left((1-o(1))\sqrt{\frac{\log X}{\log\log X}}\right).
\end{align*}
\end{theorem}

The same argument applies to other Gross--Zagier formulas once the
local hypotheses and the normalization of the height pairing are fixed.
For example, the formula of Yuan, Zhang, and Zhang
\cite{YuanZhangZhang} gives analogous applications on Shimura curves
for families satisfying its local conditions.


\medskip
\noindent\textit{Notation.}
Throughout the paper, we use the standard big $O$ and little $o$
notations, as well as the Vinogradov notation $\ll$, $\gg$, where the implied constants depend only on the subscripted parameters unless otherwise specified.
We denote $e(x):=e^{2\pi i x}$ for $x \in \mathbb{R}.$
We write $n=\square$ when $n$ is a perfect square, and
$L^{(N_0)}(s,\sym^2f)$ for the symmetric-square $L$-function with its
Euler factors at primes dividing $N_0$ omitted. Derivatives are taken
with respect to $s$ and are derivatives of the uncompleted $L$-function.
All $L$-functions are in unitary normalization, and we write the complex
variable first.

\section{Twisted first moments for \texorpdfstring{$L^{(r)}(1/2,f\otimes\chi_d)$}{central derivatives}}

In this section, we establish twisted first moments for $L^{(r)}(1/2,f\otimes\chi_d).$ We first record a shifted approximate
functional equation and then refine the argument of
Radziwi{\l}{\l} and Soundararajan \cite{MR3425386} using Shen's
adaptation \cite{Shen22} of the recursive method of Heath-Brown
\cite{HB95} and Young \cite{young}.

Fix a nonzero nonnegative function $\Phi\in C_c^\infty((1,2))$, and put
\begin{align*}
\widetilde{\Phi}(s):=\int_0^\infty\Phi(x)x^s\,dx.
\end{align*}
For $d\in\D_a^{-}$, write
\begin{align*}
L_d(s):=L(s,f\otimes\chi_d),
\qquad A_d:=\frac{\sqrt N\,|d|}{2\pi}.
\end{align*}
The completed $L$-function
\begin{align*}
\Lambda_d(s):=A_d^s\Gamma\left(s+\frac{k-1}{2}\right)L_d(s)
\end{align*}
is entire and satisfies
$\Lambda_d(s)=w_a\Lambda_d(1-s)$.
For $|z|<1/4$, define
\begin{align*}
V_z(x):=\frac1{2\pi i}\int_{(3)}
e^{s^2}\frac{\Gamma(k/2+z+s)}{\Gamma(k/2+z)}
x^{-s}\frac{ds}{s}.
\end{align*}
The shifted approximate functional equation is
\begin{gather}
L_d(1/2+z)
={}\sum_{n\geq1}\frac{\lambda_f(n)\chi_d(n)}{n^{1/2+z}}
V_z\left(\frac n{A_d}\right)\notag\\
+w_a A_d^{-2z}\frac{\Gamma(k/2-z)}{\Gamma(k/2+z)}
\sum_{n\geq1}\frac{\lambda_f(n)\chi_d(n)}{n^{1/2-z}}
V_{-z}\left(\frac n{A_d}\right);
\label{eq:AFE}
\end{gather}
see \cite[Theorem~5.3]{IK04}.
Indeed, integrate
\begin{align*}
e^{s^2}A_d^s\Gamma(k/2+z+s)L_d(1/2+z+s)/\Gamma(k/2+z)
\end{align*}
against $ds/s$, move the line from $\Re s=3$ to $\Re s=-3$, and apply
the functional equation. The residue at $s=0$ gives $L_d(1/2+z),$ and
Stirling's formula justifies differentiating \eqref{eq:AFE} any fixed
number of times near $z=0$.

To specify the local factors in the twisted first moment,
for $d\in\D_a^{-},$ define as in \cite[equation (2)]{MR3425386}
\begin{equation*}
L_a(s):=\sum_{\substack{n\geq1\\p\mid n\Rightarrow p\mid N_0}}
\frac{\lambda_f(n)\chi_d(n)}{n^s}
\qquad \text{for $\Re(s)>0$}.
\end{equation*}
This is independent of the choice of $d\in\D_a^{-}$.
For $p\nmid N_0$, put
\begin{gather*}
t_p(z):=p^{-1-2z},\\
T_p(z):=1+t_p(z)+\frac1p \cdot
(1-(\lambda_f(p)^2-2)t_p(z)+t_p(z)^2).
\end{gather*}
For $z$ near zero, let $g_z$ be the multiplicative function on integers
coprime to $N_0$ given, for $j\geq1$, by
\begin{equation}\label{eq:gz}
g_z(p^j):=
\begin{cases}
\hfil T_p(z)^{-1} & \mbox{{\normalfont if $j$ is odd,}} \\
(1+t_p(z))T_p(z)^{-1} & \mbox{{\normalfont if $j$ is even.}}
\end{cases}
\end{equation}
Set
\begin{gather}
H_a(z):=L_a(1/2+z)L^{(N_0)}(1+2z,\sym^2f)
\prod_{p\nmid N_0}
\left(1-\frac1p\right)(1-t_p(z))T_p(z).
\label{eq:Ha}
\end{gather}
The product converges absolutely for $\Re z>-1/4$.
In particular, $H_a$ is holomorphic near zero, and
\begin{equation}\label{eq:local-properties}
H_a(0)>0,\qquad g_f:=g_0,\qquad
g_f(p^j)>0,\qquad g_f(p^j)=1+O_f(p^{-1}).
\end{equation}
These are the local factors in \cite[equation (35)]{MR3425386}, with the
factors at primes dividing $N_0$ collected in $L_a$.

For $(u,N_0)=1$, write $u=u_1u_2^2$, where $u_1$ is squarefree.
Initially for $\Re z>0$, the Hecke relations give
\begin{align}
L_a(1/2+z)\prod_{p\nmid N_0}\left(1-\frac1{p^2}\right)
\sum_{\substack{n\geq1\\(n,N_0)=1\\ n u=\square}}
\frac{\lambda_f(n)}{n^{1/2+z}}
\prod_{p\mid nu}\frac p{p+1}
=
\frac{\lambda_f(u_1)}{u_1^{1/2+z}}  H_a(z)g_z(u).
\label{eq:diagonal-series}
\end{align}
For completeness, if $L_p(1+2z,\sym^2f)$ denotes the local
symmetric-square factor, then
\begin{align*}
\sum_{j\geq0}\lambda_f(p^{2j})t_p(z)^j
&=(1-t_p(z)^2)L_p(1+2z,\sym^2f),\\
\sum_{j\geq0}\lambda_f(p^{2j+1})t_p(z)^j
&=\lambda_f(p)(1-t_p(z))L_p(1+2z,\sym^2f).
\end{align*}
Applying these identities according to the parity of $v_p(u)$ proves
\eqref{eq:diagonal-series}. The combined Euler product
$H_a(z)g_z(u)$ continues holomorphically to $\Re z>-1/4$:
at primes dividing $u$, the denominators in \eqref{eq:gz} cancel
the corresponding factors in \eqref{eq:Ha}.

To state the asymptotic formula for the twisted first moments,
put $A_X:=\sqrt N X/(2\pi)$, and define
\begin{gather}
\mathcal F_w(z;X,u):=
\widetilde{\Phi}(0)u_1^{-z}H_a(z)g_z(u)
+w A_X^{-2z}\frac{\Gamma(k/2-z)}{\Gamma(k/2+z)}
\widetilde{\Phi}(-2z)u_1^zH_a(-z)g_{-z}(u).
\label{eq:moment-main}
\end{gather}

The following estimate is the main input for the resonance argument.
Its error term allows the resonator to have length $X^\theta$ for
any fixed $\theta<1/4$.

\begin{lemma}[Twisted first moment]\label{lem:twisted-first-moment}
Let $r\geq0$ be fixed. For sufficiently large $X$, uniformly for
positive integers $u$ with $(u,N_0)=1$, we have
\begin{align}
\sum_{d\in\D_a^{-}(X)}L_d^{(r)}(1/2)\chi_d(u)
\Phi\left(\frac{|d|}{X}\right)
=\,&\frac{\lambda_f(u_1)}{\sqrt{u_1}}\cdot\frac X{N_0}\cdot
\left.\frac{d^r}{dz^r}\mathcal F_{w_a}(z;X,u)\right|_{z=0} \notag\\
+\,&O_{f,r,\epsilon,\Phi}
\left((Xu)^{1/2+\epsilon}(\log X)^r\right),
\label{eq:twisted-first-moment}
\end{align}
where $u_1$ is the squarefree part of $u$. More precisely, uniformly
for $|z|\leq2/\log X$, we have
\begin{gather}
\sum_{d\in\D_a^{-}(X)}L_d(1/2+z)\chi_d(u)
\Phi\left(\frac{|d|}{X}\right)\notag\\
=\frac{\lambda_f(u_1)}{\sqrt{u_1}}\cdot\frac X{N_0} \cdot 
\mathcal F_{w_a}(z;X,u)
+O_{f,\epsilon,\Phi}\left((Xu)^{1/2+\epsilon}\right).
\label{eq:shifted-first-moment}
\end{gather}
\end{lemma}

\begin{proof}
We adapt Shen's argument \cite[Theorem~1.4]{Shen22}, which applies
the recursive method of Heath-Brown \cite{HB95} and Young \cite{young}
to quadratic twists of modular $L$-functions. We give the details of
the congruence restriction and the dependence on $u$.
Write
\begin{align*}
\mathcal M_a(z;X,u;\Phi):=
\sum_{d\in\D_a^{-}}L_d(1/2+z)\chi_d(u)
\Phi\left(\frac{|d|}{X}\right),
\end{align*}
and let $\mathcal E_a(z;X,u;\Phi)$ be the difference between this
sum and the main term in \eqref{eq:shifted-first-moment}.
For the induction, we allow shifts with
$|\Re z|\ll1/\log X$ and $|\Im z|\ll(\log X)^2$, and smooth
weights of fixed compact support. All estimates below have
polynomial dependence on the smooth norms of the weight and on
$1+|\Im z|$.

Suppose that, simultaneously for the admissible classes, we have
\begin{equation}\label{eq:induction-hypothesis}
\mathcal E_a(z;X,u;\Phi)
\ll_{f,h,\epsilon,\Phi}u^{1/2+\epsilon}X^{h+\epsilon}
\qquad (h>1/2).
\end{equation}
The case $h=1$ follows from the quadratic large sieve of Heath-Brown
\cite{HB95}, as in \cite[Corollary~2.5]{soundyoung} and
\cite[Lemma~2.5]{Shen22}, together with the functional equation.
The same argument applies after a twist by any fixed character
modulo $N_0$.

We first choose the auxiliary function in the approximate functional
equation. Following \cite[Remark~2.2]{Shen22}, put
\begin{align*}
Q_z(s)&:=\zeta(2+4z+4s)(1+4z+4s)(1-4z-4s),\\
G_z(s)&:=e^{s^2}
\frac{Q_z(s)Q_z(-s)Q_{-z}(s)Q_{-z}(-s)}
{Q_z(0)^2Q_{-z}(0)^2}.
\end{align*}
The function $G_z$ is entire and even in $s$, satisfies
$G_z(0)=1$ and $G_{-z}=G_z$, and decays rapidly in vertical
strips. Thus \eqref{eq:AFE} remains valid when $e^{s^2}$ is
replaced by $G_z(s)$ in the definition of $V_z$.
The extra factors cancel the poles arising in the contour shifts
below.

Apply M\"obius inversion in the first sum of \eqref{eq:AFE}, write
$d=m\delta^2$, and split at $\delta=Y$, where
$1\leq Y\leq\sqrt X$. Here $(\delta,N_0u)=1$, and the new
congruence is
\begin{align*}
m\equiv a\delta^{-2}\Mod{N_0}.
\end{align*}
For $\delta\leq Y$, apply Poisson summation in this progression.
Denote the zero and nonzero frequencies by $\mathcal N_0^+$ and
$\mathcal N_{\ne0}^+$, respectively. The Gauss sum calculation in
\cite[Section~10.3]{MR3425386}, followed by the first-moment bound
used in \cite[Lemma~5.3]{Shen22}, gives
\begin{equation}\label{eq:shen-nonzero}
\mathcal N_{\ne0}^+
\ll_{f,\epsilon,\Phi}u^{1/2+\epsilon}X^{1/2+\epsilon}Y.
\end{equation}
To check the congruence restriction in this estimate, separate the
factors at primes dividing $N_0$ before applying Poisson summation.
The remaining dependence on the residue of the summation variable
modulo $N_0$ is expanded in the finitely many characters modulo
$N_0$. The resulting Dirichlet series are quadratic twists of
$f\otimes\psi$, where $\psi$ is one of these characters.
At primes not dividing $N_0$, the Gauss sum calculation of
\cite[Lemma~5.2]{Shen22} applies with the additional factors
$\psi(p)^j$ in the coefficients at $p^j$. Since
$|\psi(p)|=1$, the remaining Euler product has the same bound
$\ll_{f,\epsilon}\delta^\epsilon |h|^\epsilon
u^{1/2+\epsilon}(u,h_2^2)^{1/2}$, where
$4h=h_1h_2^2$ and $h_1$ is a fundamental discriminant.
The finitely many factors at $N_0$ change only the implied constant.
Applying the quadratic large sieve to the sum over $h_1$ gives
\eqref{eq:shen-nonzero}. The change to negative fundamental discriminants
changes the Fourier kernel, without changing this bound.

For $\delta>Y$, restore squarefreeness by writing
$m=d_0b^2$ and $c=\delta b$, where $d_0$ is a negative
fundamental discriminant. Then
\begin{align*}
d_0\equiv ac^{-2}\Mod{N_0}.
\end{align*}
For $(c,N_0)=1$, the classes $a$ and $ac^{-2}$ have the same
values of the quadratic characters at primes dividing $N_0$.
Consequently, we have
\begin{align*}
L_{ac^{-2}}=L_a,\qquad H_{ac^{-2}}=H_a,
\qquad w_{ac^{-2}}=w_a.
\end{align*}
This is why the induction is taken over all admissible classes.

Put $\Phi_s(x):=x^s\Phi(x)$. Removing the Euler factors at the
primes dividing $c$ uses the identity
\begin{align*}
\prod_{p\mid c}\left(1-\frac{\lambda_f(p)\chi_{d_0}(p)}{p^v}
+\frac{\chi_{d_0}(p^2)}{p^{2v}}\right)
=\sum_{j\mid c}\sum_{\ell\mid c}
\frac{\mu(j)\mu(j{\ell})^2\lambda_f(j)\chi_{d_0}(j{\ell}^2)}{j^v{\ell}^{2v}}.
\end{align*}
Thus the part with $\delta>Y$ is
\begin{align*}
\mathcal R^+
={}&\sum_{(c,N_0u)=1}\sum_{\substack{\delta\mid c\\\delta>Y}}\mu(\delta)
\sum_{j\mid c}\sum_{{\ell}\mid c}
\frac{\mu(j)\mu(j{\ell})^2\lambda_f(j)}{j^{1/2+z}{\ell}^{1+2z}}
\frac1{2\pi i}\int_{(1/\log X)}
\frac{G_z(s)\Gamma(k/2+z+s)}{\Gamma(k/2+z)}\\
&\hspace{18mm}\times\frac{A_X^s}{j^s{\ell}^{2s}}
\mathcal M_{ac^{-2}}(z+s;X/c^2,uj{\ell}^2;\Phi_s)\frac{ds}{s}.
\end{align*}
The integral can be truncated at a fixed multiple of $(\log X)^2$
with a negligible error. Fix $0<\nu<\epsilon/10$ and apply the
induction hypothesis when $X/c^2\geq X^\nu$, where the logarithms
of the two scales are comparable. For $c>X^{(1-\nu)/2}$, the
initial first moment bound and direct bounds for the main terms give
\begin{align*}
\ll_{f,\epsilon,\Phi}
u^{1/2+\epsilon}X^{1+\epsilon/10}
\sum_{c>X^{(1-\nu)/2}}c^{-2+\epsilon/10}
\ll_{f,\epsilon,\Phi}u^{1/2+\epsilon}X^{1/2+\epsilon}.
\end{align*}
This also handles the main-term tails for $X/c^2<1$, where the
original moment vanishes for $c>\sqrt{2X}$.
Insert \eqref{eq:induction-hypothesis} on the remaining scales
and denote the contributions of the direct main term, the dual
main term, and the error by $\mathcal R_1^+$, $\mathcal R_2^+$,
and $\mathcal R_3^+$, respectively. The error satisfies
\begin{align}
\mathcal R_3^+
&\ll_{f,h,\epsilon,\Phi}
u^{1/2+\epsilon}X^{h+\epsilon}
\sum_{c>Y}c^{-2h+\epsilon}
\ll_{f,h,\epsilon,\Phi}
u^{1/2+\epsilon}X^{h+\epsilon}Y^{1-2h}.
\label{eq:shen-recursive-error}
\end{align}
Here the factors $j^{1/2}\ell$ from the new twist $uj\ell^2$ cancel
the factors in the denominator. Divisor sums and logarithms are
absorbed by reducing $\epsilon$ in the intermediate estimates.

The Euler factor identity in \cite[Lemma~6.2]{Shen22} combines
$\mathcal R_1^+$ with $\mathcal N_0^+$ and extends the
square-divisor sum to all $\delta$. It applies at every prime
not dividing $N_0$. The factors at $N_0$ remain $L_a$, since
$H_{ac^{-2}}=H_a$. Using \eqref{eq:diagonal-series}, we obtain
\begin{align*}
\mathcal N_0^++\mathcal R_1^+
=\frac{X\lambda_f(u_1)}{N_0u_1^{1/2+z}}
\frac1{2\pi i}\int_{(1)}
G_z(s)\frac{\Gamma(k/2+z+s)}{\Gamma(k/2+z)}
\widetilde\Phi(s)\left(\frac{A_X}{u_1}\right)^s
H_a(z+s)g_{z+s}(u)\frac{ds}{s}.
\end{align*}
The continuation needed to move this line to
$\Re s=-1/2+\epsilon$ follows from
\cite[Lemma~2.6]{Shen22}. More explicitly, for sufficiently
large $J$ the Euler product in $H_a(v)g_v(u)$ is factored as
\begin{align*}
H_a(v)g_v(u)
=L_a(1/2+v)L^{(N_0)}(1+2v,\sym^2f)
\frac{\zeta^{(N_0)}(2^{J+1}+2^{J+2}v)}
{\zeta^{(N_0)}(2+4v)}\mathcal Z_{a,J}(v;u),
\end{align*}
where $\mathcal Z_{a,J}$ is holomorphic and
$\ll_{f,J,\epsilon}u^\epsilon$ in
$\Re v>-1/2+\epsilon$. Here the superscript on $\zeta$ means
that the Euler factors at $N_0$ are omitted.
The denominator is cancelled by $G_z(s)$, and the other zeta
factor has no pole in the required region if $J$ is large enough.
The residue at $s=0$ is
\begin{align*}
\frac{X\lambda_f(u_1)}{N_0u_1^{1/2+z}}
\widetilde\Phi(0)H_a(z)g_z(u),
\end{align*}
and the remaining integral is
$O_{f,\epsilon,\Phi}(X^{1/2+\epsilon}u^\epsilon)$.

For $\mathcal R_2^+$, extend the sum over $\delta>Y$ to all
divisors of $c$ and use $\sum_{\delta\mid c}\mu(\delta)=\mathbf1_{c=1}$.
Move the Mellin contour in the resulting integral to
$\Re s=1/2-\epsilon$. Treat the subtracted part with
$\delta\leq Y$ on the same line, as in
\cite[equations~(6.12)--(6.14)]{Shen22}. The zero at
$1-4z-4s=0$ in $G_z(s)$ cancels the possible pole of
$\zeta(2-4z-4s)$, and the preceding factorization justifies
the remaining continuation. This gives
\begin{align*}
\mathcal R_2^+\ll_{f,\epsilon,\Phi}
X^{1/2+\epsilon}u^\epsilon Y.
\end{align*}
The second sum of \eqref{eq:AFE} is treated by replacing $z$
by $-z$ and $\Phi$ by $\Phi_{-2z}$. Its residue is the dual
term in \eqref{eq:moment-main}. Combining these estimates with
\eqref{eq:shen-nonzero} and \eqref{eq:shen-recursive-error}, we have
\begin{align*}
\mathcal E_a(z;X,u;\Phi)
\ll_{f,h,\epsilon,\Phi}
u^{1/2+\epsilon}\left(X^{1/2+\epsilon}Y
+X^{h+\epsilon}Y^{1-2h}\right).
\end{align*}
Taking $Y=X^{(2h-1)/(4h)}$ replaces $h$ in
\eqref{eq:induction-hypothesis} by $1-1/(4h)$.
Starting with $h=1$, iteration gives
$h_j=1/2+1/(2j+2)$, which tends to $1/2$.
For each fixed $\epsilon>0$, finitely many iterations prove
\eqref{eq:shifted-first-moment}. Finally, the error is holomorphic
near zero. Cauchy's formula on $|z|=1/\log X$ gives
\eqref{eq:twisted-first-moment}, and the lemma follows.
\end{proof}

\section{Proof of Theorem \ref{thm:main}}

We apply Soundararajan's resonance method \cite{MR2425151}.
Fix $0<\theta<1/4$ and $0<\eta<1$, and put
\begin{align*}
M:=X^\theta,\qquad
\mathcal L:=\eta\sqrt{\log M\log\log M}.
\end{align*}
Define a multiplicative function $\rho$, supported on squarefree
integers, by
\begin{align*}
\rho(p):=
\begin{cases}
\displaystyle\frac{\mathcal L}{\sqrt p\log p}
& \mbox{{\normalfont if $\mathcal L^2\leq p\leq\exp((\log\mathcal L)^2) \text{and } p\nmid N_0,$}}
\\
\hfil 0 & \mbox{{\normalfont otherwise,}}
\end{cases}
\end{align*}
and $\rho(p^j)=0$ for $j\geq2$. Let
\begin{equation*}
b(n):=\lambda_f(n)\rho(n),\qquad
R(d):=\sum_{n\leq M}b(n)\chi_d(n).
\end{equation*}
This is the coefficient choice of \cite[Section~2]{MR2425151},
with the factor $\lambda_f(p)$ as in \cite[Section~4]{MR4775033}.
Put
\begin{align*}
\B:=\prod_p(1+\lambda_f(p)^2\rho(p)^2),\qquad
\B_0:=\prod_p\left(1+\frac p{p+1}\lambda_f(p)^2\rho(p)^2\right).
\end{align*}
Since $|\lambda_f(p)|\leq2$ and
$\sum_p\rho(p)^2/p=o(1)$, we have
\begin{equation}\label{eq:B-comparison}
\B=(1+o(1))\B_0.
\end{equation}

We first give an upper bound for the denominator in the resonance method.

\begin{proposition}\label{prop:denominator}
We have
\begin{align*}
\sum_{d\in\D_a^{-}(X)}R(d)^2\Phi\left(\frac{|d|}{X}\right)
\ll_{f,\Phi}X\B_0.
\end{align*}
\end{proposition}

\begin{proof}
For $(n,N_0)=1$ and $n\leq M^2$, Proposition~1 of
\cite{MR3425386} gives
\begin{align*}
\sum_{d\in\D_a^{-}(X)}\chi_d(n)\Phi\left(\frac{|d|}{X}\right)
={}&c_{a,\Phi}X \cdot\mathbf1_{n=\square}
\prod_{p\mid n}\frac p{p+1}
+O_{f,\epsilon,\Phi}(X^{1/2+\epsilon}\sqrt n),
\end{align*}
where
\begin{align*}
c_{a,\Phi}:=\frac{\widetilde{\Phi}(0)}{N_0}
\prod_{p\nmid N_0}\left(1-\frac1{p^2}\right)>0.
\end{align*}
The proof of that proposition uses only the congruence and
squarefreeness conditions, so it applies to either value of $w_a$.
Expanding $R(d)^2$, we have $m_1m_2=\square$ precisely when
$m_1=m_2$, since $b$ is supported on squarefree integers.
Thus the main term is at most
\begin{align*}
c_{a,\Phi}X\sum_{m\leq M}b(m)^2\prod_{p\mid m}\frac p{p+1}
\leq c_{a,\Phi}X\B_0.
\end{align*}
By the Cauchy--Schwarz inequality, the total error is
\begin{align*}
&\ll_{f,\epsilon,\Phi}X^{1/2+\epsilon}
\left(\sum_{m\leq M}|b(m)|\sqrt m\right)^2\\
&\leq X^{1/2+\epsilon}
\left(\sum_{m\leq M}b(m)^2\right)\left(\sum_{m\leq M}m\right)
\ll X^{1/2+\epsilon}\B M^2=o(X\B_0),
\end{align*}
on choosing $\epsilon>0$ sufficiently small. Therefore, the proposition follows.
\end{proof}

Define
\begin{equation*}
\B_1:=\prod_p\left(
1+\lambda_f(p)^2\rho(p)^2g_f(p^2)
+\frac{2\lambda_f(p)^2\rho(p)g_f(p)}{\sqrt p}\right).
\end{equation*}
Using \eqref{eq:local-properties} and \eqref{eq:B-comparison}, we have that
$\B_1\gg\B_0$.

We now give a lower bound for the numerator in the resonance method.

\begin{proposition}\label{prop:numerator}
Let $w=w_a$. If either $r\geq1$, or $r=0$ and $w=+1$, then
\begin{align*}
\left|\sum_{d\in\D_a^{-}(X)}L_d^{(r)}(1/2)R(d)^2
\Phi\left(\frac{|d|}{X}\right)\right|
\gg_{f,r,\Phi}X(\log X)^r\B_1.
\end{align*}
\end{proposition}

\begin{proof}
Expanding $R(d)^2$ and applying Lemma~\ref{lem:twisted-first-moment}
with $u=m_1m_2\leq M^2<X^{1/2}$, the total error is
\begin{align}
&\ll_{f,r,\epsilon,\Phi}
X^{1/2+\epsilon}(\log X)^r
\left(\sum_{m\leq M}|b(m)|\sqrt m\right)^2\notag\\
&\ll X^{1/2+\epsilon}(\log X)^r
\left(\sum_{m\leq M}b(m)^2\right)
\left(\sum_{m\leq M}m\right)\notag\\
&\ll X^{1/2+\epsilon}(\log X)^r\B M^2
=o\bigl(X(\log X)^r\B_1\bigr),
\label{eq:numerator-error}
\end{align}
since $\theta<1/4$ and $\epsilon$ may be chosen sufficiently small.

For $u=m_1m_2=u_1u_2^2$, define
\begin{align*}
\B_M(z):=\sum_{m_1,m_2\leq M}
b(m_1)b(m_2)\frac{\lambda_f(u_1)}{u_1^{1/2+z}}g_z(u).
\end{align*}
We claim that, uniformly for $|z|\leq1/\log X$, we have
\begin{equation}\label{eq:uniform-resonator}
\B_M(z)=(1+o(1))\B_1.
\end{equation}

Without the restrictions $m_1,m_2\leq M$, the sum factors as
\begin{equation}\label{eq:deformed-product}
\B_1(z):=\prod_p\left(
1+\lambda_f(p)^2\rho(p)^2g_z(p^2)
+\frac{2\lambda_f(p)^2\rho(p)g_z(p)}{p^{1/2+z}}\right).
\end{equation}
Indeed, at each prime, its occurrence in neither, one, or both of
$m_1,m_2$ gives the three terms above. In particular, all coefficients
are nonnegative when $z=0$.

We first remove the truncations at $z=0$. Put
$\beta:=(\log\mathcal L)^{-3}$. Rankin's trick bounds the part
with $m_1>M$, divided by $\B_1$, by
\begin{align}
\exp\Bigg(-\beta\log M+(1+o(1))
\sum_p(p^\beta-1)\left(
\lambda_f(p)^2\rho(p)^2+
\frac{\lambda_f(p)^2\rho(p)}{\sqrt p}\right)\Bigg).
\label{eq:Rankin-tail}
\end{align}
Rankin--Selberg theory gives
\begin{equation*}
\sum_{p\leq y}\lambda_f(p)^2\log p\sim y;
\end{equation*}
see \cite[Lemma~3]{MR3425386}.
By partial summation, we have
\begin{equation}\label{eq:prime-sum}
\sum_{\mathcal L^2\leq p\leq\exp((\log\mathcal L)^2)}
\frac{\lambda_f(p)^2}{p\log p}
=\frac{1+o(1)}{2\log\mathcal L}.
\end{equation}
Since $\beta\log p\leq1/\log\mathcal L$ on the support of $\rho$,
we obtain
\begin{align*}
\sum_p(p^\beta-1)\lambda_f(p)^2\rho(p)^2
&=(1+o(1))\beta\mathcal L^2
\sum_{\rho(p)\ne0}\frac{\lambda_f(p)^2}{p\log p}\\
&=(\eta^2+o(1))\beta\log M.
\end{align*}
The linear term satisfies
\begin{align*}
\sum_p(p^\beta-1)\frac{\lambda_f(p)^2\rho(p)}{\sqrt p}
\ll\beta\mathcal L\sum_{\rho(p)\ne0}\frac{\lambda_f(p)^2}{p}
\ll\beta\mathcal L\log\log\mathcal L
=o(\beta\log M),
\end{align*}
and thus \eqref{eq:Rankin-tail} is
\begin{equation}\label{eq:tail-small}
\exp\bigl(-(1-\eta^2+o(1))\beta\log M\bigr)=o(1).
\end{equation}
The same bound holds for the part with $m_2>M$.

We now make this argument uniform in $z$. From \eqref{eq:gz},
for every support prime and $|z|\leq1/\log X$, we have
\begin{gather*}
g_z(p^j)=g_0(p^j)
\left(1+O_f\left(\frac{|z|\log p}{p}\right)\right)\qquad \text{for $j=1,2$},
\end{gather*}
and $p^{-z}=1+O(|z|\log p).$
Hence the sum of the absolute changes of the local coefficients
in \eqref{eq:deformed-product} is
\begin{align*}
\ll_f |z|\left(
\sum_p\frac{\lambda_f(p)^2\rho(p)^2\log p}{p}
+\sum_p\frac{\lambda_f(p)^2\rho(p)\log p}{\sqrt p}
\right)
&\ll_f\frac{1+\mathcal L\log\log\mathcal L}{\log X} \nonumber\\
&=o(1).
\end{align*}
It follows that $\B_1(z)=(1+o(1))\B_1$ uniformly on this disc.
For the absolute values of the Rankin-weighted tails, the same
calculation applies with the affected local coefficients multiplied
by $p^\beta=1+o(1)$. Thus \eqref{eq:tail-small} remains valid
uniformly in $z$, proving \eqref{eq:uniform-resonator}.

Since $H_a$ is holomorphic near zero, Cauchy's formula
and \eqref{eq:uniform-resonator} give
\begin{gather}
H_a(0)\B_M(0)=(1+o(1))H_a(0)\B_1,\notag\\
\left.\frac{d^j}{dz^j}\bigl(H_a(z)\B_M(z)\bigr)\right|_{z=0}
=o\bigl((\log X)^j\B_1\bigr)\qquad \text{for $1\leq j\leq r$}.
\label{eq:small-derivatives}
\end{gather}
The main term in the resonant moment is therefore $X/N_0$ times
the $r$-th derivative at zero of
\begin{align*}
\widetilde{\Phi}(0)H_a(z)\B_M(z)
+w A_X^{-2z}\frac{\Gamma(k/2-z)}{\Gamma(k/2+z)}
\widetilde{\Phi}(-2z)H_a(-z)\B_M(-z).
\end{align*}
For $r\geq1$, the expression \eqref{eq:small-derivatives} shows that this is
\begin{equation*}
\frac{X}{N_0}w(-2\log A_X)^r
\widetilde{\Phi}(0)H_a(0)\B_1
+o\bigl(X(\log X)^r\B_1\bigr).
\end{equation*}
Recall that the factors $\widetilde{\Phi}(0)$ and $H_a(0)$ are positive.
For $r=0$, the main term is instead
$(1+w)X\widetilde{\Phi}(0)H_a(0)\B_1/N_0+o(X\B_1)$.
Combining this with \eqref{eq:numerator-error} proves the proposition.
\end{proof}

\begin{proof}[Proof of Theorem~\ref{thm:main}]
Since $R(d)$ is real and $\Phi\geq0$, Propositions
\ref{prop:denominator} and \ref{prop:numerator} give
\begin{align}
\max_{d\in\D_a^{-}(X)}|L_d^{(r)}(1/2)|
&\geq
\frac{\left|\sum_{d\in\D_a^{-}(X)}L_d^{(r)}(1/2)R(d)^2\Phi(|d|/X)\right|}
{\sum_{d\in\D_a^{-}(X)}R(d)^2\Phi(|d|/X)}\notag\\
&\gg_{f,r}(\log X)^r\frac{\B_1}{\B_0}.
\label{eq:resonance-ratio}
\end{align}
Using \eqref{eq:local-properties} and the definitions of $\B_0$ and
$\B_1$, we have
\begin{align}
\log\frac{\B_1}{\B_0}
&=2\sum_p\frac{\lambda_f(p)^2\rho(p)}{\sqrt p}
+o\left(\frac{\mathcal L}{\log\mathcal L}\right)\notag\\
&=2\mathcal L
\sum_{\mathcal L^2\leq p\leq\exp((\log\mathcal L)^2)}
\frac{\lambda_f(p)^2}{p\log p}
+o\left(\frac{\mathcal L}{\log\mathcal L}\right)\notag\\
&=(1+o(1))\frac{\mathcal L}{\log\mathcal L}.
\label{eq:euler-gain}
\end{align}
Indeed, the errors in the first line are bounded by a constant
multiple of
\begin{align*}
\sum_{\rho(p)\ne0}
\left(\frac{\rho(p)^2}{p}
+\frac{\rho(p)^3}{\sqrt p}
+\frac{\rho(p)}{p^{3/2}}\right)
=o\left(\frac{\mathcal L}{\log\mathcal L}\right),
\end{align*}
and the last line of \eqref{eq:euler-gain} follows from
\eqref{eq:prime-sum}.
Now
\begin{align*}
\frac{\mathcal L}{\log\mathcal L}
&=(2\eta+o(1))\sqrt{\frac{\log M}{\log\log M}}\\
&=(2\eta\sqrt\theta+o(1))
\sqrt{\frac{\log X}{\log\log X}}.
\end{align*}
For any fixed $0<C<1$, choose $\eta<1$ and
$\theta<1/4$ so that $2\eta\sqrt\theta>C$.
Then \eqref{eq:resonance-ratio} and \eqref{eq:euler-gain} imply
the desired bound for all sufficiently large $X,$ proving the theorem.
\end{proof}

\section{Classical Gross--Zagier formula}\label{sec:GZ}

Let $E/\Q$ be an elliptic curve of conductor $N$ and analytic rank $r_{\rm an}(E)$
at most one. By the modularity theorem \cite{BCDT01}, there is a
normalized newform $f_E\in S_2^*(\Gamma_0(N))$ such that
$L(s,E)=L(s,f_E)$ in the unitary normalization, and thus
\begin{equation*}
\operatorname{ord}_{s=1/2}L(s,f_E)\in\{0,1\}.
\end{equation*}

For every $d\in\D_1^{-}$, each prime dividing $N$ splits in
$K_d:=\Q(\sqrt d)$. Indeed, $d\equiv1\Mod p$ for every odd
$p\mid N$, so $\chi_d(p)=1$, while $d\equiv1\Mod8$ gives the
same conclusion at $p=2;$ see \cite[Chapter~I, Section~3]{GZ86}).

Let $H_d$ be the Hilbert class field of $K_d$, and choose an ideal
$\mathfrak n_d\subseteq\mathcal O_{K_d}$ with
$\mathcal O_{K_d}/\mathfrak n_d\cong\Z/N\Z$.
The cyclic isogeny
\begin{align*}
\C/\mathcal O_{K_d}\longrightarrow\C/\mathfrak n_d^{-1}
\end{align*}
defines a Heegner point $x_d\in X_0(N)(H_d)$.
Fix a modular parametrization
\begin{align*}
\pi:X_0(N)\longrightarrow E,\qquad \pi(i\infty)=O,
\end{align*}
defined over $\Q$, and set
\begin{equation*}
P_d:=\Tr_{H_d/K_d}\pi(x_d)\in E(K_d).
\end{equation*}
Artin formalism gives
\begin{equation}\label{eq:Artin-factorization}
L(s,E/K_d)=L(s,f_E)L(s,f_E\otimes\chi_d).
\end{equation}

The Gross--Zagier formula
\cite[Theorem~2.1]{GZ86} gives
\begin{equation*}
\widehat h(P_d)=c_Eu_d^2\sqrt{|d|}\,L'(1/2,E/K_d),
\end{equation*}
where $c_E>0$ depends only on $E$ and the fixed parametrization,
and $u_d:=\#\mathcal O_{K_d}^{\times}/2$.
For $d<-4$, we have $u_d=1$. In particular, we have
\begin{equation}\label{eq:GZ-asymp}
\widehat h(P_d)\asymp_E
\sqrt{|d|}\,\left|L'(1/2,E/K_d)\right|.
\end{equation}
Also, for $d\in\D_1^{-}$, the expression \eqref{eq:root-number} gives
\begin{equation}\label{eq:Heegner-root-number}
w(f_E\otimes\chi_d)=w(f_E)\chi_d(-N)=-w(f_E),
\end{equation}
since $\chi_d(N)=1$ and $\chi_d(-1)=-1$.

\subsection{Proof of Theorem \ref{thm:heegner}}

To prove the theorem, we consider separately the cases of analytic rank zero and one.

\subsubsection{Analytic rank zero}

Suppose that $L(1/2,f_E)\ne0$. Then $w(f_E)=+1$, so
\eqref{eq:Heegner-root-number} gives
$w(f_E\otimes\chi_d)=-1$ and hence
$L(1/2,f_E\otimes\chi_d)=0$ for $d\in\D_1^{-}$.
Differentiating \eqref{eq:Artin-factorization}, we obtain
\begin{equation}\label{eq:rank-zero-factorization}
L'(1/2,E/K_d)=L(1/2,f_E)L'(1/2,f_E\otimes\chi_d).
\end{equation}
Applying Theorem~\ref{thm:main} with $r=1$ and $a=1$, and using
\eqref{eq:GZ-asymp} and \eqref{eq:rank-zero-factorization} together
with $|d|\geq X$, we obtain
\begin{equation}\label{eq:rank-zero-height}
\max_{d\in\D_1^{-}(X)}\widehat h(P_d)
\gg_E\sqrt X\log X \cdot
\exp\left(C\sqrt{\frac{\log X}{\log\log X}}\right)
\end{equation}
for every fixed $0<C<1$.

\subsubsection{Analytic rank one}

Suppose that $L(1/2,f_E)=0$ and $L'(1/2,f_E)\ne0$.
Then $w(f_E)=-1$, so the twists in $\D_1^{-}$ have root number
$+1$. Differentiating \eqref{eq:Artin-factorization} now gives
\begin{equation}\label{eq:rank-one-factorization}
L'(1/2,E/K_d)=L'(1/2,f_E)L(1/2,f_E\otimes\chi_d).
\end{equation}
Applying Theorem~\ref{thm:main} with $r=0$ and $a=1$, and using
\eqref{eq:GZ-asymp} and \eqref{eq:rank-one-factorization} together
with $|d|\geq X$, we obtain
\begin{equation}\label{eq:rank-one-height}
\max_{d\in\D_1^{-}(X)}\widehat h(P_d)
\gg_E\sqrt X
\exp\left(C\sqrt{\frac{\log X}{\log\log X}}\right)
\end{equation}
for every fixed $0<C<1$.
Finally, combining \eqref{eq:rank-zero-height} and \eqref{eq:rank-one-height}
proves Theorem~\ref{thm:heegner}.

\section{Zhang's Gross--Zagier formula}\label{sec:Zhang}

Let $f\in S_k^*(\Gamma_0(N))$ have analytic rank $r_f$ at most one.
For every $d\in\D_1^{-}$, the field $K_d=\Q(\sqrt d)$ satisfies
the Heegner hypothesis. We retain the notation $H_d$ and $x_d$ from
Section~\ref{sec:GZ}, and write $h_d=[H_d:K_d]$.
We first describe the Heegner cycles and the normalization of their heights.
For the construction, we refer to \cite[Sections~0.1, 0.3, and~4.1]{shouwuzhang}.

After passing to an auxiliary full level, let $\mathcal E$ denote the
universal generalized elliptic curve over the corresponding modular curve.
A smooth projective desingularization of its $(k-2)$-fold fiber product
is a Kuga--Sato variety $W_{k-2}$ of dimension $k-1$.
The construction at level $\Gamma_0(N)$ is obtained from this auxiliary
level by the normalized pullback in \cite[Section~4.1]{shouwuzhang}.

Put $m=k/2$. Above a Heegner point $x_d$, the graph of complex
multiplication gives a cycle in the $(k-2)$-fold product of the
corresponding elliptic curve. Applying the alternating projector and
normalizing its fiber self-intersection to $(-1)^{m-1}$ gives Zhang's
cycle $s_{m,d}(x_d)$ of codimension $m$ on $W_{k-2}$.
For $k\geq4$, these are homologically trivial cycles with real coefficients.
When $k=2$, put $W_0=X_0(N)$ and use the degree-zero divisor
$(x_d)-(i\infty)$.
Write $\operatorname{CH}^{m}(W_{k-2})_{\mathbb R,0}$ for the group of
homologically trivial codimension-$m$ cycles modulo rational equivalence,
with real coefficients.

The height pairing on these cycles is defined by arithmetic intersection
as in \cite{GilleSoule,shouwuzhang}.
Let $V_d$ be the real space generated by the Hecke translates of
$s_{m,d}(x_d^\sigma)$ for $\sigma\in\operatorname{Gal}(H_d/K_d)$,
and let $V_d^0$ be its quotient by the radical of the height pairing.
By \cite[Theorem~0.3.1]{shouwuzhang}, the Hecke action on $V_d^0$
admits the eigenspace decomposition used below. For $k=2$, this
is the usual Hecke decomposition on $J_0(N)(H_d)\otimes\mathbb R$.
Define
\begin{align*}
\Delta_{f,d}:=
\left(\sum_{\sigma\in\operatorname{Gal}(H_d/K_d)}
s_{m,d}(x_d^\sigma)\right)_f\in V_d^0,
\end{align*}
where the subscript denotes the $f$-isotypic component.
A representative in the cycle group gives the same self-pairing.
We use Zhang's sign convention and normalize the pairing by the field
degree. Thus, if $\langle\ ,\ \rangle_{\mathrm Z,H_d}$ denotes
Zhang's pairing, we put
\begin{align*}
\langle\ ,\ \rangle_{\rm{BB}}:=
\frac{1}{2h_d}\langle\ ,\ \rangle_{\mathrm Z,H_d}.
\end{align*}

As in the elliptic curve case, Artin formalism gives
\begin{equation}\label{eq:cycle-Artin}
L(s,f/K_d)=L(s,f)L(s,f\otimes\chi_d).
\end{equation}
With the above normalizations, Zhang's formula
\cite[Corollary~0.3.2]{shouwuzhang} for the trivial class group
character, and \cite[Theorem~6.3]{GZ86} when $k=2$, give
\begin{align*}
\langle\Delta_{f,d},\Delta_{f,d}\rangle_{\rm{BB}}
=c_fu_d^2\sqrt{|d|}\,L'(1/2,f/K_d),
\end{align*}
where $c_f>0$ depends only on $f$ and the fixed geometric normalizations.
The factor $h_d$ in Zhang's formula is canceled by our normalization
of the pairing. Since $u_d=1$ for $d<-4$, we have
\begin{equation}\label{eqn:Zhang}
\left|\langle\Delta_{f,d},\Delta_{f,d}\rangle_{\rm{BB}}\right|
\asymp_f\sqrt{|d|}\,\left|L'(1/2,f/K_d)\right|.
\end{equation}
In higher weight, positivity of this self-pairing is not assumed.
This accounts for the absolute value in Theorem~\ref{thm:heegnerCycle}.
Also, the expression \eqref{eq:root-number} gives
\begin{equation}\label{eq:cycle-root-number}
w(f\otimes\chi_d)=-w(f)
\end{equation}
for every $d\in\D_1^{-}$.

\subsection{Proof of Theorem \ref{thm:heegnerCycle}}

Similarly, we consider separately the cases of analytic rank
zero and one.

\subsubsection{Analytic rank zero}

Suppose that $L(1/2,f)\ne0$. Then $w(f)=+1$, and
\eqref{eq:cycle-root-number} gives $w(f\otimes\chi_d)=-1$.
Differentiating \eqref{eq:cycle-Artin}, we obtain
\begin{align*}
L'(1/2,f/K_d)=L(1/2,f)L'(1/2,f\otimes\chi_d).
\end{align*}
Applying Theorem~\ref{thm:main} with $r=1$ and $a=1$, and using
\eqref{eqn:Zhang}, we have
\begin{equation}\label{eq:cycle-rank-zero-height}
\max_{d\in\D_1^{-}(X)}
\left|\langle\Delta_{f,d},\Delta_{f,d}\rangle_{\rm{BB}}\right|
\gg_f\sqrt X\log X \cdot
\exp\left(C\sqrt{\frac{\log X}{\log\log X}}\right)
\end{equation}
for every fixed $0<C<1$.

\subsubsection{Analytic rank one}

Suppose that $L(1/2,f)=0$ and $L'(1/2,f)\ne0$.
Then $w(f)=-1$, and the twists in $\D_1^{-}$ have root number $+1$.
Differentiating \eqref{eq:cycle-Artin} now gives
\begin{align*}
L'(1/2,f/K_d)=L'(1/2,f)L(1/2,f\otimes\chi_d).
\end{align*}
Applying Theorem~\ref{thm:main} with $r=0$ and $a=1$, and using
\eqref{eqn:Zhang}, we have
\begin{equation}\label{eq:cycle-rank-one-height}
\max_{d\in\D_1^{-}(X)}
\left|\langle\Delta_{f,d},\Delta_{f,d}\rangle_{\rm{BB}}\right|
\gg_f\sqrt X
\exp\left(C\sqrt{\frac{\log X}{\log\log X}}\right)
\end{equation}
for every fixed $0<C<1$.
Combining \eqref{eq:cycle-rank-zero-height} and
\eqref{eq:cycle-rank-one-height} proves Theorem~\ref{thm:heegnerCycle}.

\section*{Acknowledgements}
The authors would like to thank Hazem Hassan and Gopal Maiti for helpful discussions.

\clearpage

\printbibliography

@misc{dong2026extremevaluesquadraticdirichlet,
      title={Extreme values of quadratic Dirichlet $L$-functions}, 
      author={Z. Dong and W. Wang and H. Zhang and S. Zhao},
      year={2026},
      eprint={2607.20408},
      archivePrefix={arXiv},
   %   primaryClass={math.NT},
   %   url={https://arxiv.org/abs/2607.20408}, 
}

@article {AL78,
    AUTHOR = {Atkin, A. O. L. and Li, W. C. W.},
     TITLE = {Twists of newforms and pseudo-eigenvalues of {$W$}-operators},
   JOURNAL = {Invent. Math.},
  FJOURNAL = {Inventiones Mathematicae},
    VOLUME = {48},
      YEAR = {1978},
    NUMBER = {3},
     PAGES = {221--243},
  %    ISSN = {0020-9910,1432-1297},
 %  MRCLASS = {10D12},
 % MRNUMBER = {508986},
%MRREVIEWER = {R.\ A.\ Rankin},
  %     DOI = {10.1007/BF01390245},
  %     URL = {https://doi.org/10.1007/BF01390245},
}

@article {BCDT01,
    AUTHOR = {Breuil, C. and Conrad, B. and Diamond, F. and
              Taylor, R.},
     TITLE = {On the modularity of elliptic curves over {$\mathbf Q$}: wild
              3-adic exercises},
   JOURNAL = {J. Amer. Math. Soc.},
  FJOURNAL = {Journal of the American Mathematical Society},
    VOLUME = {14},
      YEAR = {2001},
    NUMBER = {4},
     PAGES = {843--939},
  %    ISSN = {0894-0347,1088-6834},
 %  MRCLASS = {11G05 (11F80 11G07 14G35)},
%  MRNUMBER = {1839918},
%MRREVIEWER = {Karl\ Rubin},
 %      DOI = {10.1090/S0894-0347-01-00370-8},
 %      URL = {https://doi.org/10.1090/S0894-0347-01-00370-8},
}

@article {MR4958486,
    AUTHOR = {Darbar, P. and Maiti, G.},
     TITLE = {Large values of quadratic {D}irichlet {$L$}-functions},
   JOURNAL = {Math. Ann.},
  FJOURNAL = {Mathematische Annalen},
    VOLUME = {392},
      YEAR = {2025},
    NUMBER = {4},
     PAGES = {4573--4605},
  %    ISSN = {0025-5831,1432-1807},
 %  MRCLASS = {11M06 (11L40 11M20 11N37)},
%  MRNUMBER = {4958486},
%MRREVIEWER = {Christoph\ Aistleitner},
  %     DOI = {10.1007/s00208-025-03187-6},
%       URL = {https://doi.org/10.1007/s00208-025-03187-6},
}

@article {GZ86,
    AUTHOR = {Gross, B. H. and Zagier, D. B.},
     TITLE = {Heegner points and derivatives of {$L$}-series},
   JOURNAL = {Invent. Math.},
  FJOURNAL = {Inventiones Mathematicae},
    VOLUME = {84},
      YEAR = {1986},
    NUMBER = {2},
     PAGES = {225--320},
 %     ISSN = {0020-9910,1432-1297},
 %  MRCLASS = {11G40 (11F11 11G05 14G10)},
%  MRNUMBER = {833192},
%MRREVIEWER = {Loren\ D.\ Olson},
 %      DOI = {10.1007/BF01388809},
%       URL = {https://doi.org/10.1007/BF01388809},
}

@article {MR4775033,
    AUTHOR = {Gun, S. and Kohnen, W. and Soundararajan, K.},
     TITLE = {Large {F}ourier coefficients of half-integer weight modular
              forms},
   JOURNAL = {Amer. J. Math.},
  FJOURNAL = {American Journal of Mathematics},
    VOLUME = {146},
      YEAR = {2024},
    NUMBER = {4},
     PAGES = {1169--1191},
   %   ISSN = {0002-9327,1080-6377},
  % MRCLASS = {11F37 (11F30)},
 % MRNUMBER = {4775033},
%MRREVIEWER = {Ariel\ M.\ Pacetti},
  %     DOI = {10.1353/ajm.2024.a932437},
  %     URL = {https://doi.org/10.1353/ajm.2024.a932437},
}

@article {HB95,
    AUTHOR = {Heath-Brown, D. R.},
     TITLE = {A mean value estimate for real character sums},
   JOURNAL = {Acta Arith.},
  FJOURNAL = {Acta Arithmetica},
    VOLUME = {72},
      YEAR = {1995},
    NUMBER = {3},
     PAGES = {235--275},
  %    ISSN = {0065-1036,1730-6264},
%   MRCLASS = {11L40 (11M06 11M26)},
%  MRNUMBER = {1347489},
%MRREVIEWER = {Matti\ Jutila},
 %      DOI = {10.4064/aa-72-3-235-275},
  %     URL = {https://doi.org/10.4064/aa-72-3-235-275},
}

@article {MR4670148,
    AUTHOR = {Hua, S. and Huang, B.},
     TITLE = {Extreme central {$L$}-values of almost prime quadratic twists
              of elliptic curves},
   JOURNAL = {Sci. China Math.},
  FJOURNAL = {Science China. Mathematics},
    VOLUME = {66},
      YEAR = {2023},
    NUMBER = {12},
     PAGES = {2755--2766},
   %   ISSN = {1674-7283,1869-1862},
   %MRCLASS = {11F67},
 % MRNUMBER = {4670148},
%MRREVIEWER = {Shaoyun\ Yi},
  %     DOI = {10.1007/s11425-022-2216-y},
%       URL = {https://doi.org/10.1007/s11425-022-2216-y},
}

@book {IK04,
    AUTHOR = {Iwaniec, H. and Kowalski, E.},
     TITLE = {Analytic number theory},
    SERIES = {American Mathematical Society Colloquium Publications},
    VOLUME = {53},
 PUBLISHER = {American Mathematical Society, Providence, RI},
      YEAR = {2004},
     PAGES = {xii+615},
  %    ISBN = {0-8218-3633-1},
%   MRCLASS = {11-02 (11Fxx 11Lxx 11Mxx 11Nxx)},
 % MRNUMBER = {2061214},
%MRREVIEWER = {K.\ Soundararajan},
%       DOI = {10.1090/coll/053},
%       URL = {https://doi.org/10.1090/coll/053},
}

@article {MR3425386,
    AUTHOR = {Radziwi\l\l, M. and Soundararajan, K.},
     TITLE = {Moments and distribution of central {$L$}-values of quadratic
              twists of elliptic curves},
   JOURNAL = {Invent. Math.},
  FJOURNAL = {Inventiones Mathematicae},
    VOLUME = {202},
      YEAR = {2015},
    NUMBER = {3},
     PAGES = {1029--1068},
   %   ISSN = {0020-9910,1432-1297},
 %  MRCLASS = {11M41 (11G05)},
%  MRNUMBER = {3425386},
%MRREVIEWER = {D.\ R.\ Heath-Brown},
   %    DOI = {10.1007/s00222-015-0582-z},
  %     URL = {https://doi.org/10.1007/s00222-015-0582-%z},
}

@article {Shen22,
    AUTHOR = {Shen, Q.},
     TITLE = {The first moment of quadratic twists of modular
              {$L$}-functions},
   JOURNAL = {Acta Arith.},
  FJOURNAL = {Acta Arithmetica},
    VOLUME = {206},
      YEAR = {2022},
    NUMBER = {4},
     PAGES = {313--337},
   %  ISSN = {0065-1036,1730-6264},
  % MRCLASS = {11M06 (11F67)},
%  MRNUMBER = {4528422},
%MRREVIEWER = {Keiju\ Sono},
  %     DOI = {10.4064/aa211207-7-11},
   %    URL = {https://doi.org/10.4064/aa211207-7-11},
}

@article {MR2425151,
    AUTHOR = {Soundararajan, K.},
     TITLE = {Extreme values of zeta and {$L$}-functions},
   JOURNAL = {Math. Ann.},
  FJOURNAL = {Mathematische Annalen},
    VOLUME = {342},
      YEAR = {2008},
    NUMBER = {2},
     PAGES = {467--486},
    %  ISSN = {0025-5831,1432-1807},
  % MRCLASS = {11M06 (11N56)},
%  MRNUMBER = {2425151},
%MRREVIEWER = {D.\ R.\ Heath-Brown},
  %     DOI = {10.1007/s00208-008-0243-2},
 %      URL = {https://doi.org/10.1007/s00208-008-0243-2},
}

@article {shouwuzhang,
    AUTHOR = {Zhang, S.},
     TITLE = {Heights of {H}eegner cycles and derivatives of {$L$}-series},
   JOURNAL = {Invent. Math.},
  FJOURNAL = {Inventiones Mathematicae},
    VOLUME = {130},
      YEAR = {1997},
    NUMBER = {1},
     PAGES = {99--152},
   %   ISSN = {0020-9910,1432-1297},
 %  MRCLASS = {11G40 (11F66 14G10 14G40)},
 % MRNUMBER = {1471887},
%MRREVIEWER = {David\ Harari},
   %    DOI = {10.1007/s002220050179},
    %   URL = {https://doi-org.lib-%ezproxy.concordia.ca/10.1007/s002220050179},
}

@article {GilleSoule,
    AUTHOR = {Gillet, H. and Soul\'e, C.},
     TITLE = {Arithmetic intersection theory},
   JOURNAL = {Inst. Hautes \'Etudes Sci. Publ. Math.},
  FJOURNAL = {Institut des Hautes \'Etudes Scientifiques. Publications
              Math\'ematiques},
    NUMBER = {72},
      YEAR = {1990},
     PAGES = {93--174},
  %    ISSN = {0073-8301,1618-1913},
%   MRCLASS = {14G40 (14C17 32C30)},
%  MRNUMBER = {1087394},
%MRREVIEWER = {I.\ Dolgachev},
 %      URL = {http://www.numdam.org/item?id=PMIHES_1990__72__93_0},
}

@article {soundyoung,
    AUTHOR = {Soundararajan, K. and Young, M. P.},
     TITLE = {The second moment of quadratic twists of modular
              {$L$}-functions},
   JOURNAL = {J. Eur. Math. Soc. (JEMS)},
  FJOURNAL = {Journal of the European Mathematical Society (JEMS)},
    VOLUME = {12},
      YEAR = {2010},
    NUMBER = {5},
     PAGES = {1097--1116},
   %   ISSN = {1435-9855,1435-9863},
 %  MRCLASS = {11F66 (11F67)},
%  MRNUMBER = {2677611},
%MRREVIEWER = {D.\ R.\ Heath-Brown},
  %     DOI = {10.4171/JEMS/224},
 %      URL = {https://doi-org.lib-%ezproxy.concordia.ca/10.4171/JEMS/224},
}

@book {YuanZhangZhang,
    AUTHOR = {Yuan, X. and Zhang, S.-W. and Zhang, W.},
     TITLE = {The {G}ross-{Z}agier formula on {S}himura curves},
    SERIES = {Annals of Mathematics Studies},
    VOLUME = {184},
 PUBLISHER = {Princeton University Press, Princeton, NJ},
      YEAR = {2013},
     PAGES = {x+256},
   %   ISBN = {978-0-691-15592-0},
 %  MRCLASS = {11G18 (11F70 14G35)},
%  MRNUMBER = {3237437},
%MRREVIEWER = {Ernest\ Hunter\ Brooks},
}

@article {young,
    AUTHOR = {Young, M. P.},
     TITLE = {The first moment of quadratic {D}irichlet {$L$}-functions},
   JOURNAL = {Acta Arith.},
  FJOURNAL = {Acta Arithmetica},
    VOLUME = {138},
      YEAR = {2009},
    NUMBER = {1},
     PAGES = {73--99},
%      ISSN = {0065-1036,1730-6264},
%   MRCLASS = {11M06 (11L40)},
%  MRNUMBER = {2529465},
%MRREVIEWER = {Arnaud\ Chadozeau},
%       DOI = {10.4064/aa138-1-4},
%       URL = {https://doi-org.lib-ezproxy.concordia.ca/10.4064/aa138-%1-4},
}

\end{document}